\documentclass[12pt]{amsart}

\usepackage[margin=1in]{geometry}
\usepackage{amsmath,amssymb,amsthm}
\usepackage{setspace}
\usepackage{natbib}
\usepackage{hyperref}

\newtheorem{theorem}{Theorem}[section]
\newtheorem{proposition}[theorem]{Proposition}
\newtheorem{lemma}[theorem]{Lemma}
\newtheorem{corollary}[theorem]{Corollary}
\theoremstyle{definition}
\newtheorem{remark}[theorem]{Remark}
\newtheorem{example}[theorem]{Example}

\newcommand{\Res}{\operatorname{Res}}
\newcommand{\Disc}{\operatorname{Disc}}
\DeclareMathOperator{\len}{\ell}

\title[Resultant of an equivariant system for $G(r,n)$]
{Resultant of an equivariant polynomial system with respect to the reflection group $G(r,n)$}

\author{Sonagnon Julien Owolabi $^1$, Ibrahim Nonkan\'{e}$^2$ and
Joel Tossa$^3$ } 
\address{ $^{1,3}$Institut de Math\'{e}matiques et de Sciences Physiques, IMSP, Universit\'{e} d'Abomey-Calavi, B\'{e}nin ,$^2$D\'epartement d'\'Economie et Math\'ematiques Appliqu\'ees, Universit\'e Thomas Sankara, Ouagadougou, Burkina Faso}
\email{$^1$julien.owolabi@imsp-auc.org, $^2$ibrahim.nonkane@uts.bf, $^3$joel.tossa@imsp-uac.org}

\date{\today}

\begin{document}

\begin{abstract}
We consider systems of homogeneous multivariate polynomials equivariant under the complex reflection group $G(r,n) = (\mathbb{Z}/r\mathbb{Z})^n \rtimes S_n$. Using divided differences indexed by partitions of $n$, we establish a decomposition formula expressing the resultant of such a system as a product of resultants of smaller, partition-indexed subsystems. Combining this with the classical resultant--discriminant relation, we show that the discriminant of a $G(r,n)$-invariant homogeneous polynomial splits explicitly into a product of resultants of smaller subsystems, considerably easier to compute; we illustrate both decompositions with worked examples.
\end{abstract}


\maketitle
\footnote{Corresponding author: Ibrahim Nonkan\'{e}, email address of the corresponding author: Ibrahim.nonkane@uts.bf} 
\section{Introduction}

Resultants and discriminants are two of the most fundamental tools of elimination theory: the resultant of $n$ homogeneous polynomials in $n$ variables detects whether the corresponding projective hypersurfaces have a common point, while the discriminant of a single homogeneous polynomial detects whether the hypersurface it defines is singular. Both can, in principle, be computed for an arbitrary system through Macaulay's classical determinantal formula \citep{Macaulay1902}; in practice, however, the size of the relevant matrices grows combinatorially with the number of variables and the degree, so that a direct computation quickly becomes intractable. At the same time, polynomial systems arising in applications very often carry a natural symmetry. This is the case, for instance, of the equilibrium equations for $n$ interacting point vortices in the plane, which are invariant under the symmetric group $S_n$ permuting the vortices \citep{FaugereSvartz2012}; it is also the case, more generally, whenever a system of equations is built, by construction, so as to be compatible with the action of a finite group of linear symmetries of the ambient space. Exploiting such a symmetry to reduce the computation of a resultant or a discriminant to that of several resultants of much smaller systems is therefore both a natural and a computationally useful question. Bus\'e and Karasoulou \citep{BuseKarasoulou2016} answered it for systems equivariant under the symmetric group $S_n$; the present note extends their approach to systems equivariant under the larger family of complex reflection groups $G(r,n)$ defined below, and draws from it an explicit decomposition formula for the discriminant of a $G(r,n)$-invariant homogeneous polynomial.

Let $S_n$ be the group of permutations of the set of variables $\{x_1,\dots,x_n\}$ and $\mathbb{Z}/r\mathbb{Z}$ be the cyclic group of order $r$ which acts on the $x_i$ by a primitive $r$th root of unity.

The reflection group $G(r,n)$ is the semi-direct product of $(\mathbb{Z}/r\mathbb{Z})^n$ with $S_n$, written $(\mathbb{Z}/r\mathbb{Z})^n \rtimes S_n$, where $(\mathbb{Z}/r\mathbb{Z})^n$ is the direct product of $n$ copies of $\mathbb{Z}/r\mathbb{Z}$. Let $\xi$ be a primitive $r$-th root of $1$.
\[
(\mathbb{Z}/r\mathbb{Z})^n \rtimes S_n = \{ (\xi^{i_1},\dots,\xi^{i_n};\sigma) \mid i_k \in \{0,1,\dots,r-1\},\ \sigma \in S_n \}.
\]
Let $\mathbb{C}[x_1,\dots,x_n]$ be the ring of polynomials in $n$ indeterminates on which the group $G(r,n)$ acts as follows: for $f \in \mathbb{C}[x_1,\dots,x_n]$ and $(\xi^{i_1},\dots,\xi^{i_n};\sigma) \in G(r,n)$,
\[
(\xi^{i_1},\dots,\xi^{i_n};\sigma)f = f(\xi^{i_{\sigma(1)}}x_{\sigma(1)},\dots,\xi^{i_{\sigma(n)}}x_{\sigma(n)}).
\]

A polynomial system $A := \{f^{\{1\}}, f^{\{2\}}, \dots, f^{\{n\}}\}$ is said to be \emph{equivariant} with respect to a finite group $G$ if for all $g \in G$ and all $i = 1,\dots,n$, we have $g(f^{\{i\}}) \in A$. In other words $A$ is globally stable under the finite group $G$. Whenever, as is the case for $G(r,n)$ below, $G$ acts on the index set $\{1,\dots,n\}$ through its permutation quotient, we will always use this finer, labelled form of equivariance, requiring $g(f^{\{i\}}) = f^{\{g\cdot i\}}$ for the specific index $g\cdot i$ rather than merely $g(f^{\{i\}}) \in A$; this is made precise for $G(r,n)$ just below.

Consider a system of $n$ homogeneous polynomials $f^{\{1\}}, f^{\{2\}}, \dots, f^{\{n\}}$ in $\mathbb{C}[x_1,\dots,x_n]$ of the same degree $d$, equivariant with respect to the reflection group $G(r,n)$. More precisely, we assume that for any integer $i \in \{1,\dots,n\}$ and $(\xi^{i_1},\dots,\xi^{i_n};\sigma) \in G(r,n)$,
\[
(\xi^{i_1},\dots,\xi^{i_n};\sigma)f^{\{i\}} := f^{\{i\}}(\xi^{i_{\sigma(1)}}x_{\sigma(1)},\dots,\xi^{i_{\sigma(n)}}x_{\sigma(n)}) = f^{\{\sigma(i)\}}(x_1,\dots,x_n).
\]

In this work, we study the resultant of such systems and obtain a decomposition formula for the discriminant of an invariant multivariate homogeneous polynomial under the reflection group $G(r,n)$.

Note that $G(r,n)$ is the group denoted $G(r,1,n)$ in the Shephard--Todd classification of complex reflection groups; it specializes, for $r=1$, to the symmetric group $S_n$ itself, and, for $r=2$, to the hyperoctahedral group of signed permutations, so that the results below apply in particular to these two classical and widely used groups. Our main tool is a decomposition formula, established by Bus\'e and Karasoulou \citep{BuseKarasoulou2016}, for the resultant of an $S_n$-equivariant homogeneous polynomial system; we extend it to $G(r,n)$-equivariant systems, and we combine it with the relation between the resultant and the discriminant established by Bus\'e and Jouanolou \citep{BuseJouanolou2014} to obtain our decomposition formula for the discriminant. In the particular case $r=1$, i.e.\ for $S_n$-invariant polynomials, related decomposition formulas for the discriminant were obtained by Perminov and Shakirov \citep{PerminovShakirov2009}.

The paper is organized as follows. In Section~\ref{sec:resultant}, we establish the decomposition formula for the resultant of a $G(r,n)$-equivariant polynomial system, first in terms of the polynomials $P^{\{i\}}$ (Theorem~\ref{thm:main1}) and then, via divided differences indexed by partitions of $n$, directly in terms of the original system $f^{\{1\}}, \dots, f^{\{n\}}$ (Theorem~\ref{thm:main2}); several worked numerical examples illustrate the resulting simplification. In Section~\ref{sec:discriminant}, we apply this decomposition to the discriminant of a $G(r,n)$-invariant homogeneous polynomial, first via the relation between the resultant and the discriminant (Theorem~\ref{thm:disc1}), and then, more explicitly, via a direct differentiation argument (Theorem~\ref{thm:disc2}). We conclude in Section~\ref{sec:conclusion} with a summary and some perspectives for further work.

\section{Resultant of a $G(r,n)$-equivariant polynomial system}
\label{sec:resultant}

In this section, we consider a polynomial system of $n$ homogeneous polynomials $f^{\{1\}}, f^{\{2\}}, \dots, f^{\{n\}}$ in $\mathbb{C}[x_1,\dots,x_n]$ of the same degree $d$. They are of the form
\[
f^{\{i\}}(x_1,\dots,x_n) := \sum_{|\alpha|=d} a_{i,\alpha}\, x^{\alpha}, \qquad i = 1,\dots,n,
\]
where $\alpha = (\alpha_1,\dots,\alpha_n) \in \mathbb{N}^n$, $|\alpha| := \sum_{i=1}^n \alpha_i$, $x^\alpha = x_1^{\alpha_1}\cdots x_n^{\alpha_n}$.

This system is equivariant with respect to the reflection group $G(r,n)$ if $r$ divides $d$, for $i \in \{1,\dots,n\}$, $r$ divides every $\alpha_i$ appearing in $f^{\{i\}}$, and for all $\sigma \in S_n$, $\sigma(f^{\{i\}}) = f^{\{\sigma(i)\}}$.

\begin{remark}
This condition is in fact equivalent to the $G(r,n)$-equivariance of the system $\{f^{\{1\}},\dots,f^{\{n\}}\}$ in the sense of the Introduction, i.e.\ to requiring $(\xi^{i_1},\dots,\xi^{i_n};\sigma)f^{\{i\}} = f^{\{\sigma(i)\}}$ for every $(\xi^{i_1},\dots,\xi^{i_n};\sigma)\in G(r,n)$ and every $i \in \{1,\dots,n\}$.

($\Rightarrow$) Taking $\sigma = \mathrm{id}$ and letting $(i_1,\dots,i_n)$ range over $\{0,\dots,r-1\}^n$ shows that every $f^{\{i\}}$ is invariant under every diagonal torsion $x_k \mapsto \xi^{i_k}x_k$; comparing the coefficient of each monomial $x^\alpha$ of $f^{\{i\}}$ on both sides forces $\xi^{\alpha_k}=1$ for every $k$, i.e.\ $r$ divides every $\alpha_k$. Taking instead $i_1=\dots=i_n=0$ and $\sigma \in S_n$ arbitrary gives $\sigma(f^{\{i\}}) = f^{\{\sigma(i)\}}$.

($\Leftarrow$) Conversely, assume $r$ divides every exponent of every $f^{\{i\}}$ and $\sigma(f^{\{i\}}) = f^{\{\sigma(i)\}}$ for all $\sigma \in S_n$. Since every exponent of $f^{\{i\}}$ is divisible by $r$, we may write $f^{\{i\}}(x_1,\dots,x_n) = Q^{\{i\}}(x_1^r,\dots,x_n^r)$ for some polynomial $Q^{\{i\}}$ (this is the polynomial $P^{\{i\}}$ introduced below). For any $(\xi^{i_1},\dots,\xi^{i_n};\sigma)\in G(r,n)$, using $\xi^r=1$,
\begin{align*}
f^{\{i\}}(\xi^{i_{\sigma(1)}}x_{\sigma(1)},\dots,\xi^{i_{\sigma(n)}}x_{\sigma(n)}) &= Q^{\{i\}}\big((\xi^{i_{\sigma(1)}}x_{\sigma(1)})^r,\dots\big) \\
&= Q^{\{i\}}(x_{\sigma(1)}^r,\dots,x_{\sigma(n)}^r) = f^{\{i\}}(x_{\sigma(1)},\dots,x_{\sigma(n)}),
\end{align*}
the last equality being the assumed $\sigma$-equivariance. Hence $(\xi^{i_1},\dots,\xi^{i_n};\sigma)f^{\{i\}} = f^{\{\sigma(i)\}}$, as required.
\end{remark}

Note also that the requirement ``$r$ divides $d$'' is redundant: it follows automatically from ``$r$ divides $\alpha_i$ for all $i$'', since $d = |\alpha| = \sum_{i=1}^n \alpha_i$ is then itself a multiple of $r$; we keep it for emphasis.

Since every exponent $\alpha_i$ appearing in $f^{\{i\}}$ is divisible by $r$, for all $i \in \{1,\dots,n\}$,
\[
f^{\{i\}}(x_1,\dots,x_n) = P^{\{i\}}(x_1^r, x_2^r,\dots,x_n^r)
\]
where $P^{\{1\}}, P^{\{2\}}, \dots, P^{\{n\}}$ are $n$ homogeneous polynomials in $\mathbb{C}[x_1,\dots,x_n]$ of the same degree $d/r$.

For any integer $i \in \{1,\dots,n\}$ and $(\xi^{i_1},\dots,\xi^{i_n};\sigma) \in G(r,n)$, using $\xi^r = 1$ we get
\begin{align*}
(\xi^{i_1},\dots,\xi^{i_n};\sigma)f^{\{i\}} &:= f^{\{i\}}(\xi^{i_{\sigma(1)}}x_{\sigma(1)},\dots,\xi^{i_{\sigma(n)}}x_{\sigma(n)}) \\
&= P^{\{i\}}\big((\xi^{i_{\sigma(1)}}x_{\sigma(1)})^r,\dots,(\xi^{i_{\sigma(n)}}x_{\sigma(n)})^r\big) \\
&= P^{\{i\}}(x_{\sigma(1)}^r,\dots,x_{\sigma(n)}^r),
\end{align*}
so that the torsion part of the group element plays no role in this computation. On the other hand, since the system $f^{\{1\}},\dots,f^{\{n\}}$ is equivariant with respect to $G(r,n)$, the same left-hand side equals $f^{\{\sigma(i)\}}(x_1,\dots,x_n) = P^{\{\sigma(i)\}}(x_1^r,\dots,x_n^r)$. Combining the two computations, we obtain
\[
P^{\{i\}}(x_{\sigma(1)}^r,\dots,x_{\sigma(n)}^r) = P^{\{\sigma(i)\}}(x_1^r,\dots,x_n^r)
\]
for all $x_1,\dots,x_n$. The $\mathbb{C}$-algebra homomorphism $\mathbb{C}[y_1,\dots,y_n] \to \mathbb{C}[x_1,\dots,x_n]$, $y_k \mapsto x_k^r$, sends distinct monomials $y^\beta$ to distinct monomials $x^{r\beta}$, and is therefore injective. Since both sides of the identity above are the image, under this map, of the polynomials $P^{\{i\}}(y_{\sigma(1)},\dots,y_{\sigma(n)})$ and $P^{\{\sigma(i)\}}(y_1,\dots,y_n)$ respectively, injectivity allows us to cancel the substitution $y_k = x_k^r$ and conclude that already
\[
P^{\{i\}}(x_{\sigma(1)},\dots,x_{\sigma(n)}) = P^{\{\sigma(i)\}}(x_1,\dots,x_n)
\]
holds as an identity of polynomials in $\mathbb{C}[x_1,\dots,x_n]$. Thus $P^{\{1\}}, P^{\{2\}},\dots,P^{\{n\}}$ form an $S_n$-equivariant system of homogeneous polynomials of degree $d/r$.

Throughout the rest of the paper we use freely, without further comment, the classical properties of the resultant of $n$ homogeneous polynomials in $n$ variables: homogeneity in the coefficients of each polynomial, multiplicativity, invariance up to sign under a linear change of variables and under permutation of the polynomials, and the normalization $\Res(x_1^{d_1},\dots,x_n^{d_n}) = 1$ for any degrees $d_1,\dots,d_n$; see \citep[\S 5]{Jouanolou1991}, \citep{Jouanolou1997}, or the textbook account in \citep[Chapter 3]{CoxLittleOShea2005}.

\subsection{Divided differences and partitions}

Let $f^{\{1\}},\dots,f^{\{n\}}$ be $n$ homogeneous polynomials in $\mathbb{C}[x_1,\dots,x_n]$ of the same degree $d > 1$. We assume that for all $i \in \{1,\dots,n\}$, $f^{\{i\}}(x_1,\dots,x_n) = P^{\{i\}}(x_1^r,\dots,x_n^r)$ with $r > 1$.

The divided differences of the polynomials $P^{\{1\}},\dots,P^{\{n\}}$ are defined by
\[
P^{\{i_1,\dots,i_k\}} = \frac{P^{\{i_1,\dots,i_{k-1}\}} - P^{\{i_1,\dots,i_{k-2},i_k\}}}{x_{i_{k-1}} - x_{i_k}}
\]
for any given set of distinct integers $I := \{i_1,\dots,i_k\} \subset [n]$.

A priori, this recursive formula only defines $P^I$ as a rational function of $x_1,\dots,x_n$, for a fixed ordered tuple $(i_1,\dots,i_k)$. The next lemma shows that, thanks to the $S_n$-equivariance of $P^{\{1\}},\dots,P^{\{n\}}$ established above, $P^I$ is in fact a well-defined homogeneous polynomial depending only on the set $I$, and not on the order in which its elements are listed; this justifies the set notation $P^I = P^{\{i_1,\dots,i_k\}}$ used throughout the paper.

\begin{lemma}
\label{lem:dividedwelldefined}
For all $i \ne j$ in $\{1,\dots,n\}$, $x_i - x_j$ divides $P^{\{i\}} - P^{\{j\}}$. Consequently, for every $I \subset [n]$ with $|I|=k$, $P^I$ is a homogeneous polynomial of degree $d/r - k + 1$, which depends only on the set $I$ and not on the order in which its elements are listed.
\end{lemma}

\begin{proof}
Let $i \ne j$ and let $\tau = (i\, j) \in S_n$ be the transposition exchanging $i$ and $j$. Since $P^{\{1\}},\dots,P^{\{n\}}$ form an $S_n$-equivariant system, $\tau(P^{\{i\}}) = P^{\{\tau(i)\}} = P^{\{j\}}$, i.e.\ $P^{\{i\}}(x_1,\dots,x_n)$ with $x_i$ and $x_j$ exchanged equals $P^{\{j\}}(x_1,\dots,x_n)$. Hence $P^{\{i\}} - P^{\{j\}} = P^{\{i\}} - \tau(P^{\{i\}})$ vanishes when $x_i = x_j$, so $x_i - x_j$ divides $P^{\{i\}} - P^{\{j\}}$; this is the case $|I|=2$.

We argue by induction on $|I|=k \geqslant 2$ that $P^I$ is a well-defined polynomial, independent of the chosen order on $I$; the case $k=2$ is what we just proved. For $J \subset [n]$ and distinct $i,j,k' \notin J$ with $|J|=k-2$, a direct computation from the recursive definition of divided differences gives
\[
(x_i - x_j)P^{J\cup\{i,j\}} - (x_i - x_{k'})P^{J\cup\{i,k'\}} + (x_j - x_{k'})P^{J\cup\{j,k'\}} = 0,
\]
which can be rewritten as
\[
(x_i - x_{k'})\big(P^{J\cup\{i,j\}} - P^{J\cup\{i,k'\}}\big) = (x_j - x_{k'})\big(P^{J\cup\{i,j\}} - P^{J\cup\{j,k'\}}\big).
\]
By the induction hypothesis, $P^{J\cup\{i,j\}}$, $P^{J\cup\{i,k'\}}$ and $P^{J\cup\{j,k'\}}$ are well-defined polynomials depending only on their index sets; the identity above then forces $x_j - x_{k'}$ to divide $P^{J\cup\{i,j\}} - P^{J\cup\{i,k'\}}$, so that $P^{J\cup\{i,j,k'\}} = \big(P^{J\cup\{i,j\}} - P^{J\cup\{i,k'\}}\big)/(x_j - x_{k'})$ is itself a well-defined polynomial; a symmetric computation shows it does not depend on which of $i,j,k'$ is singled out last, and hence, together with the induction hypothesis, on no ordering at all of $I = J \cup \{i,j,k'\}$. This settles the induction (this is a classical property of divided differences; see e.g.\ \citep{Jouanolou1991}).

Finally, since each $P^{\{i\}}$ is homogeneous of degree $d/r$ and each divided difference lowers the degree by exactly one, $P^I$ is homogeneous of degree $d/r - |I| + 1$.
\end{proof}

\begin{corollary}
\label{cor:transitive}
For every $\sigma \in S_n$ and every $I \subset [n]$, $\sigma(P^I) = P^{\sigma(I)}$. Consequently, if $|I| = d/r + 1$ (so that $P^I$ is homogeneous of degree $0$, i.e.\ a constant, by Lemma~\ref{lem:dividedwelldefined}), this constant does not depend on the choice of the subset $I \subset [n]$ of size $d/r+1$: since $S_n$ acts transitively on such subsets and fixes every constant, $P^I = P^J$ for all $I, J \subset [n]$ with $|I|=|J|=d/r+1$. This is what justifies writing $P^{\{1,\dots,d/r+1\}}$, as in Theorem~\ref{thm:main1} below, without specifying which particular subset of that size is meant.
\end{corollary}

\begin{proof}
We argue by induction on $|I|$. For $|I|=1$, this is exactly the hypothesis $\sigma(P^{\{i\}}) = P^{\{\sigma(i)\}}$. For $|I|=k \geqslant 2$, write, using Lemma~\ref{lem:dividedwelldefined}, $I = J \cup \{i_{k-1},i_k\}$ with $|J|=k-2$ and $P^I = \big(P^{J\cup\{i_{k-1}\}} - P^{J\cup\{i_{k-2},i_k\}}\big)/(x_{i_{k-1}}-x_{i_k})$. Applying the ring automorphism $\sigma$ (which sends $x_l \mapsto x_{\sigma(l)}$) to this identity and using the induction hypothesis on the two divided differences of size $k-1$ gives
\[
\sigma(P^I) = \frac{P^{\sigma(J)\cup\{\sigma(i_{k-1})\}} - P^{\sigma(J)\cup\{\sigma(i_{k-2}),\sigma(i_k)\}}}{x_{\sigma(i_{k-1})} - x_{\sigma(i_k)}} = P^{\sigma(I)},
\]
which completes the induction.
\end{proof}

The divided differences of the polynomials $f^{\{1\}},\dots,f^{\{n\}}$ are recursively defined by $\Delta f^{\{i\}} := f^{\{i\}}$ for all $i=1,\dots,n$ and
\begin{equation}
\label{eq:deltadef}
\Delta f^{\{i_1,\dots,i_k\}} = \frac{\Delta f^{\{i_1,\dots,i_{k-1}\}} - \Delta f^{\{i_1,\dots,i_{k-2},i_k\}}}{x_{i_{k-1}}^r - x_{i_k}^r}.
\end{equation}

\begin{remark}
\label{rem:deltawelldefined}
The same conclusion holds for the divided differences $\Delta f^{\{i_1,\dots,i_k\}}$ of $f^{\{1\}},\dots,f^{\{n\}}$: since $f^{\{i\}}(x_1,\dots,x_n) = P^{\{i\}}(x_1^r,\dots,x_n^r)$, the divisibility $x_i - x_j \mid P^{\{i\}} - P^{\{j\}}$ of Lemma~\ref{lem:dividedwelldefined}, evaluated at $(x_1^r,\dots,x_n^r)$, gives $x_i^r - x_j^r \mid f^{\{i\}} - f^{\{j\}}$ for all $i \ne j$; the same induction as in the proof of Lemma~\ref{lem:dividedwelldefined}, with $x_i - x_j$ replaced by $x_i^r - x_j^r$ throughout, then shows that $\Delta f^I$ is a well-defined homogeneous polynomial, depending only on the set $I$, of degree $d - r(|I|-1)$ for every $I \subset [n]$.
\end{remark}

Let $\lambda = (\lambda_1,\lambda_2,\dots,\lambda_{\len(\lambda)})$ be a sequence such that $\lambda_1 \geqslant \dots \geqslant \lambda_{\len(\lambda)} > 0$. When $\sum_{i=1}^{\len(\lambda)} \lambda_i = n$, we say that $\lambda$ is a partition of $n$, and write $\lambda \vdash n$.

Given a partition $\lambda \vdash n$, its associated multinomial coefficient is
\begin{equation}
\label{eq:multinom}
m_\lambda := \frac{1}{\prod_{j=1}^n s_j!} \binom{n}{\lambda_1,\lambda_2,\dots,\lambda_{\len(\lambda)}} = \frac{n!}{\big(\prod_{j=1}^n s_j!\big)\,\lambda_1!\lambda_2!\cdots \lambda_{\len(\lambda)}!},
\end{equation}
where $s_j$ denotes the number of parts of $\lambda$ equal to $j$, for $j \in [n]$.

We consider the following homomorphism of algebras:
\begin{equation}
\label{eq:rholambda}
\begin{aligned}
\rho_\lambda : \mathbb{C}[x_1,\dots,x_n] &\longrightarrow \mathbb{C}[y_1,\dots,y_{\len(\lambda)}] \\
P(x_1,\dots,x_n) &\longmapsto P(\underbrace{y_1,\dots,y_1}_{\lambda_1},\dots,\underbrace{y_{\len(\lambda)},\dots,y_{\len(\lambda)}}_{\lambda_{\len(\lambda)}}),
\end{aligned}
\end{equation}
where $y_1,\dots,y_{\len(\lambda)}$ are new indeterminates. Explicitly, $x_1,\dots,x_{\lambda_1}$ are sent to $y_1$, $x_{\lambda_1+1},\dots,x_{\lambda_1+\lambda_2}$ are sent to $y_2$, and so on, the $i$-th block $\{\lambda_1+\dots+\lambda_{i-1}+1,\dots,\lambda_1+\dots+\lambda_i\}$ of $\lambda_i$ consecutive indices being sent to $y_i$. Since $\rho_\lambda$ sends each $x_k$ to a single variable, it preserves total degree: for any homogeneous polynomial $P$ of degree $e$, $\rho_\lambda(P)$ is again homogeneous of degree $e$ (or is zero).

For any integer $i \in [\len(\lambda)]$, $P^{\{i\}}_\lambda := \rho_\lambda(P^{\{j\}}(x_1,\dots,x_n))$, where $j \in [n]$ is such that $\rho_\lambda(x_j) = y_i$.

\begin{remark}
\label{rem:rholambdaI}
More generally, given $I = \{i_1,\dots,i_k\} \subset [n]$, choose $J = \{j_1,\dots,j_k\} \subset [\len(\lambda)]$ such that $\rho_\lambda(x_{i_s}) = y_{j_s}$ for all $s \in [k]$; this is possible whenever $i_1,\dots,i_k$ belong to $k$ distinct blocks of $\lambda$, i.e.\ whenever $|J|=|I|$. In that case, since $\rho_\lambda$ is a ring homomorphism and, by Lemma~\ref{lem:dividedwelldefined}, $x_{i_s} - x_{i_t}$ divides $P^{\{i_s\}} - P^{\{i_t\}}$ for $s \ne t$, applying $\rho_\lambda$ to the recursive definition of $P^I$ shows that $P^I_\lambda := \rho_\lambda\big(P^I(x_1,\dots,x_n)\big)$ is again well defined and coincides with the divided difference of $P^{\{1\}}_\lambda,\dots,P^{\{\len(\lambda)\}}_\lambda$ indexed by $J$. This is the formula we use below, always with index sets $I = \{1,2,\dots,\len(\lambda)\}$ (one representative per block), for which the condition $|J|=|I|$ is automatically satisfied.
\end{remark}

\subsection{The decomposition formula}

The first step of the proof of Theorem~\ref{thm:main1} below rewrites $\Res(f^{\{1\}},\dots,f^{\{n\}})$ in terms of $\Res(P^{\{1\}},\dots,P^{\{n\}})$. We isolate this as a separate lemma.

\begin{lemma}
\label{lem:powersubstitution}
Let $n \geqslant 2$ and $r \geqslant 1$ be integers, and let $Q_1,\dots,Q_n \in \mathbb{C}[y_1,\dots,y_n]$ be homogeneous polynomials of respective degrees $d_1',\dots,d_n' \geqslant 1$, not necessarily equal. Set $f_i(x_1,\dots,x_n) := Q_i(x_1^r,\dots,x_n^r)$, homogeneous of degree $r d_i'$. Then
\[
\Res(f_1,\dots,f_n) = \Res(Q_1,\dots,Q_n)^{r^{n-1}}.
\]
\end{lemma}

\begin{proof}
Write $\Phi := \Res_x(f_1,\dots,f_n)$ and $\Psi := \Res_y(Q_1,\dots,Q_n)$, viewed as polynomials in the coefficients of $Q_1,\dots,Q_n$. (The coefficients of $f_i$ coincide with those of $Q_i$, only reindexed via $\beta \mapsto r\beta$, since $f_i(x) = \sum_\beta a_{i,\beta} x^{r\beta}$ whenever $Q_i(y) = \sum_\beta a_{i,\beta} y^\beta$.)

\emph{Same vanishing locus.} The map $\pi : \mathbb{P}^{n-1}(\mathbb{C}) \to \mathbb{P}^{n-1}(\mathbb{C})$, $[x_1:\dots:x_n] \mapsto [x_1^r:\dots:x_n^r]$, is a well-defined morphism (its coordinates $x_1^r,\dots,x_n^r$ have no common zero but the origin) which is onto, because $\mathbb{C}$ is algebraically closed: given $[y_1:\dots:y_n]$, choosing an $r$-th root $x_i$ of each $y_i$ gives $\pi([x_1:\dots:x_n]) = [y_1:\dots:y_n]$. Since $f_i = Q_i \circ \pi$, a point $x$ is a common zero of $f_1,\dots,f_n$ if and only if $\pi(x)$ is a common zero of $Q_1,\dots,Q_n$; as $\pi$ is onto, $f_1,\dots,f_n$ have a common zero in $\mathbb{P}^{n-1}(\mathbb{C})$ if and only if $Q_1,\dots,Q_n$ do. Hence, for every specialization of the coefficients $a_{i,\beta}$, $\Phi = 0$ if and only if $\Psi = 0$; that is, $\Phi$ and $\Psi$ vanish on exactly the same subset of the coefficient space. (This argument does not use the degrees $d_i'$ being equal.)

\emph{Proportionality.} The universal resultant $\Psi$ is an irreducible polynomial in the coefficients $a_{i,\beta}$ (a classical fact about resultants; see e.g.\ \citep{Jouanolou1991}), and $\Phi \ne 0$ (e.g.\ the system $f_1,\dots,f_n$ associated with $Q_i = y_i^{d_i'}$ has no nontrivial common zero). Together with the previous paragraph, this forces $\Phi = c\cdot \Psi^m$ for some nonzero constant $c \in \mathbb{C}$ and some integer $m \geqslant 1$.

\emph{Degree count.} By the classical homogeneity property of the resultant, for any fixed index $i$, $\Phi$ is homogeneous in the coefficients of $f_i$ (equivalently, in the coefficients of $Q_i$, since they coincide) of degree $\prod_{k \ne i}\deg f_k = \prod_{k\ne i}(rd_k') = r^{n-1}\prod_{k\ne i}d_k'$, while $\Psi$ is homogeneous in the coefficients of $Q_i$ of degree $\prod_{k\ne i}d_k'$. Comparing degrees in $\Phi = c\Psi^m$ gives $r^{n-1}\prod_{k\ne i}d_k' = m\prod_{k\ne i}d_k'$, i.e.\ $m = r^{n-1}$; this holds for every choice of $i$, as it must, and in particular does not require the $d_k'$'s to be equal.

\emph{Normalization.} Specializing to $Q_i = y_i^{d_i'}$ for every $i$ gives $f_i = x_i^{rd_i'}$, so $\Phi = \Res(x_1^{rd_1'},\dots,x_n^{rd_n'}) = 1$ and $\Psi = \Res(y_1^{d_1'},\dots,y_n^{d_n'}) = 1$ by the defining normalization of the resultant, which holds for arbitrary (not necessarily equal) degrees; hence $1 = c\cdot 1^{r^{n-1}}$, so $c=1$.

Altogether, $\Res(f_1,\dots,f_n) = \Phi = \Psi^{r^{n-1}} = \Res(Q_1,\dots,Q_n)^{r^{n-1}}$.
\end{proof}

\begin{remark}
Only the equal-degree case $d_1' = \dots = d_n' = d'$ is needed for Theorem~\ref{thm:main1} just below. The general, unequal-degree statement of Lemma~\ref{lem:powersubstitution} will be used again, with $n$ replaced by $\len(\lambda)$, in the proof of Proposition~\ref{prop:deltaP}, where the $\len(\lambda)$ polynomials involved have the decreasing degrees $d/r, d/r-1,\dots,d/r-\len(\lambda)+1$.
\end{remark}

\begin{theorem}
\label{thm:main1}
Assume $n \geqslant 2$ and a system of $n$ homogeneous polynomials $f^{\{1\}},f^{\{2\}},\dots,f^{\{n\}}$ in $\mathbb{C}[x_1,\dots,x_n]$ of the same degree $d$, equivariant with respect to the reflection group $G(r,n)$. For all $i \in \{1,\dots,n\}$, write $f^{\{i\}}(x_1,\dots,x_n) = P^{\{i\}}(x_1^r,\dots,x_n^r)$.
\begin{itemize}
\item If $d/r \geqslant n$ then
\[
\Res\big(f^{\{1\}},\dots,f^{\{n\}}\big) = \prod_{\lambda \vdash n} \Res\Big(P^{\{1\}}_\lambda, P^{\{1,2\}}_\lambda,\dots,P^{\{1,2,\dots,\len(\lambda)\}}_\lambda\Big)^{m_\lambda \times r^{n-1}}.
\]
\item If $d/r < n$ then
\[
\Res\big(f^{\{1\}},\dots,f^{\{n\}}\big) = \Big(P^{\{1,\dots,d/r+1\}}\Big)^{m_0 \times r^{n-1}} \times \prod_{\substack{\lambda \vdash n \\ \len(\lambda) \leqslant d/r}} \Res\Big(P^{\{1\}}_\lambda,\dots,P^{\{1,2,\dots,\len(\lambda)\}}_\lambda\Big)^{m_\lambda\times r^{n-1}}
\]
where
\[
m_0 := n\binom{d/r}{n-1} - \sum_{\substack{\lambda \vdash n\\ \len(\lambda)\leqslant d/r}} m_\lambda \left(\sum_{j=1}^{\len(\lambda)} \frac{(d/r)(d/r-1)\cdots(d/r-\len(\lambda)+1)}{d/r - j + 1}\right).
\]
\end{itemize}
\end{theorem}

\begin{proof}
Applying Lemma~\ref{lem:powersubstitution} with $Q_i = P^{\{i\}}$, exponent $r$, and $d' = d/r$ gives directly
\begin{equation}
\label{eq:step1}
\Res\big(f^{\{1\}}(x),\dots,f^{\{n\}}(x)\big) = \Res\big(P^{\{1\}}(x),\dots,P^{\{n\}}(x)\big)^{r^{n-1}}.
\end{equation}
As shown in Section~\ref{sec:resultant}, $P^{\{1\}},\dots,P^{\{n\}}$ form an $S_n$-equivariant system of homogeneous polynomials of degree $d/r$. A decomposition formula for the resultant of an $S_n$-equivariant homogeneous polynomial system is known (\citep{BuseKarasoulou2016}), and applies here since $P^{\{1\}},\dots,P^{\{n\}}$ satisfy exactly its hypotheses. We get:
\begin{itemize}
\item If $d/r \geqslant n$ then $\displaystyle \Res\big(P^{\{1\}},\dots,P^{\{n\}}\big) = \prod_{\lambda \vdash n} \Res\Big(P^{\{1\}}_\lambda,\dots,P^{\{1,2,\dots,\len(\lambda)\}}_\lambda\Big)^{m_\lambda}$;
\item If $d/r < n$ then
\[
\Res\big(P^{\{1\}},\dots,P^{\{n\}}\big) = \Big(P^{\{1,\dots,d/r+1\}}\Big)^{m_0} \times \prod_{\substack{\lambda\vdash n\\ \len(\lambda)\leqslant d/r}} \Res\Big(P^{\{1\}}_\lambda,\dots,P^{\{1,2,\dots,\len(\lambda)\}}_\lambda\Big)^{m_\lambda},
\]
\end{itemize}
with $m_0$ as displayed in the statement. Raising both sides of each of these two identities to the power $r^{n-1}$, and substituting into \eqref{eq:step1}, yields the two claimed formulas.
\end{proof}

All the numerical examples below have been computed using Macaulay's classical determinantal formula for the resultant \citep{Macaulay1902}.

\begin{example}
\label{ex:first}
Consider the system of three homogeneous polynomials
\[
\begin{cases}
f^{\{1\}} = a x_1^9 + (b+c)x_1^3x_2^3x_3^3 + c x_1^6x_2^3 + c x_1^6x_3^3 \\
f^{\{2\}} = a x_2^9 + (b+c)x_1^3x_2^3x_3^3 + c x_2^6x_1^3 + c x_2^6x_3^3 \\
f^{\{3\}} = a x_3^9 + (b+c)x_1^3x_2^3x_3^3 + c x_3^6x_1^3 + c x_3^6x_2^3
\end{cases}
\]
This system is equivariant with respect to the reflection group $G(3,3)$. We get
\[
\begin{cases}
f^{\{1\}}(x_1,x_2,x_3) = P^{\{1\}}(x_1^3,x_2^3,x_3^3) \\
f^{\{2\}}(x_1,x_2,x_3) = P^{\{2\}}(x_1^3,x_2^3,x_3^3) \\
f^{\{3\}}(x_1,x_2,x_3) = P^{\{3\}}(x_1^3,x_2^3,x_3^3)
\end{cases}
\qquad\text{with}\qquad
\begin{cases}
P^{\{1\}} = ax_1^3 + (b+c)x_1x_2x_3 + cx_1^2x_2 + cx_1^2x_3 \\
P^{\{2\}} = ax_2^3 + (b+c)x_1x_2x_3 + cx_2^2x_1 + cx_2^2x_3 \\
P^{\{3\}} = ax_3^3 + (b+c)x_1x_2x_3 + cx_3^2x_1 + cx_3^2x_2
\end{cases}
\]
so that $P^{\{1\}},P^{\{2\}},P^{\{3\}}$ form an $S_3$-equivariant system of homogeneous polynomials of degree $d/r = 9/3 = 3$, and
\begin{align*}
\Res\big(f^{\{1\}},f^{\{2\}},f^{\{3\}}\big) &= \Res\big(P^{\{1\}}_{(3)}\big)^{m_{(3)}\times 9} \times \Res\big(P^{\{1\}}_{(2,1)},P^{\{1,2\}}_{(2,1)}\big)^{m_{(2,1)}\times 9} \\
&\quad \times \Res\big(P^{\{1\}}_{(1,1,1)},P^{\{1,2\}}_{(1,1,1)},P^{\{1,2,3\}}_{(1,1,1)}\big)^{m_{(1,1,1)}\times 9}.
\end{align*}
We compute the divided differences
\[
P^{\{1,2\}} = \frac{P^{\{1\}}-P^{\{2\}}}{x_1-x_2} = ax_1^2 + ax_1x_2 + ax_2^2 + cx_1x_2 + cx_1x_3 + cx_2x_3,
\]
\[
P^{\{1,3\}} = \frac{P^{\{1\}}-P^{\{3\}}}{x_1-x_3} = ax_1^2 + ax_1x_3 + ax_3^2 + cx_1x_2 + cx_1x_3 + cx_2x_3,
\]
\[
P^{\{1,2,3\}} = \frac{P^{\{1,2\}}-P^{\{1,3\}}}{x_2-x_3} = ax_1+ax_2+ax_3,
\]
and, under $\rho_\lambda$,
\[
P^{\{1\}}_{(3)} = (a+b+3c)x_1^3, \qquad P^{\{1\}}_{(2,1)} = (a+c)x_1^3 + (b+2c)x_1^2x_2, \qquad P^{\{1,2\}}_{(2,1)} = (a+c)x_1^2 + ax_2^2 + (a+2c)x_1x_2.
\]
Hence
\[
\Res\big(P^{\{1\}}_{(3)}\big) = a+b+3c, \qquad \Res\big(P^{\{1\}}_{(2,1)}, P^{\{1,2\}}_{(2,1)}\big) = a^2(a+c)(a^2-ab-ac+b^2+2bc),
\]
\[
\Res\big(P^{\{1\}}_{(1,1,1)},P^{\{1,2\}}_{(1,1,1)},P^{\{1,2,3\}}_{(1,1,1)}\big) = a^6(a-c)^3(a+b)^2,
\]
so that
\[
\Res\big(f^{\{1\}},f^{\{2\}},f^{\{3\}}\big) = a^{108}(a-c)^{27}(a+c)^{27}(a+b)^{18}(a+b+3c)^9\big(a^2-ab-ac+b^2+2bc\big)^{27}.
\]
\end{example}

\begin{proposition}
\label{prop:deltaP}
Assume $n \geqslant 2$ and a system of $n$ homogeneous polynomials $f^{\{1\}},\dots,f^{\{n\}}$ in $\mathbb{C}[x_1,\dots,x_n]$ of the same degree $d$, equivariant with respect to $G(r,n)$, with $f^{\{i\}}(x_1,\dots,x_n) = P^{\{i\}}(x_1^r,\dots,x_n^r)$. Let $\lambda \vdash n$ be a partition of $n$ such that $\len(\lambda) \leqslant d/r$ (so that, by Lemma~\ref{lem:dividedwelldefined}, the polynomial $P^{\{1,\dots,j\}}_\lambda$ has degree $d/r-j+1 \geqslant 1$ for every $1 \leqslant j \leqslant \len(\lambda)$). Then
\begin{equation}
\label{eq:propdeltaP}
\Res\Big(\Delta f^{\{1\}}_\lambda,\dots,\Delta f^{\{1,\dots,\len(\lambda)\}}_\lambda\Big) = \Res\Big(P^{\{1\}}_\lambda,\dots,P^{\{1,\dots,\len(\lambda)\}}_\lambda\Big)^{r^{\len(\lambda)-1}}.
\end{equation}
\end{proposition}

\begin{proof}
From $P^{\{1,2\}}(x_1,\dots,x_n) = \dfrac{P^{\{1\}}-P^{\{2\}}}{x_1-x_2}$ we get
\[
P^{\{1,2\}}(x_1^r,\dots,x_n^r) = \frac{P^{\{1\}}(x_1^r,\dots,x_n^r) - P^{\{2\}}(x_1^r,\dots,x_n^r)}{x_1^r-x_2^r} = \frac{f^{\{1\}}-f^{\{2\}}}{x_1^r-x_2^r} = \Delta f^{\{1,2\}},
\]
and, by the same token, $P^{\{1,2,3\}}(x_1^r,\dots,x_n^r) = \Delta f^{\{1,2,3\}}$, \dots, $P^{\{1,\dots,\len(\lambda)\}}(x_1^r,\dots,x_n^r) = \Delta f^{\{1,\dots,\len(\lambda)\}}$. For $i \in [\len(\lambda)]$,
\[
P^{\{1,\dots,i\}}_\lambda(y_1^r,\dots,y_{\len(\lambda)}^r) = \Delta f^{\{1,\dots,i\}}_\lambda(y_1,\dots,y_{\len(\lambda)}),
\]
where, for each $i \in [\len(\lambda)]$, $y_i = \rho_\lambda(x_j)$ for (any) $j$ in the $i$-th block of $\lambda$.

By Lemma~\ref{lem:dividedwelldefined}, the polynomials $P^{\{1\}}_\lambda,\dots,P^{\{1,\dots,\len(\lambda)\}}_\lambda$ are homogeneous of respective degrees $d/r,\ d/r-1,\dots,d/r-\len(\lambda)+1$, all $\geqslant 1$ since $\len(\lambda)\leqslant d/r$ by hypothesis. Lemma~\ref{lem:powersubstitution} therefore applies to the $\len(\lambda)$ polynomials $Q_i := P^{\{1,\dots,i\}}_\lambda(y_1,\dots,y_{\len(\lambda)})$, $i=1,\dots,\len(\lambda)$ (of respective degrees $d_i' = d/r-i+1$, not necessarily equal), with $n$ there replaced by $\len(\lambda)$. Using the identity above to rewrite each $Q_i(y_1^r,\dots,y_{\len(\lambda)}^r)$ as $\Delta f^{\{1,\dots,i\}}_\lambda(y_1,\dots,y_{\len(\lambda)})$, Lemma~\ref{lem:powersubstitution} yields
\[
\Res\Big(\Delta f^{\{1\}}_\lambda,\dots,\Delta f^{\{1,\dots,\len(\lambda)\}}_\lambda\Big) = \Res\big(Q_1(y^r),\dots,Q_{\len(\lambda)}(y^r)\big) = \Res\Big(P^{\{1\}}_\lambda,\dots,P^{\{1,\dots,\len(\lambda)\}}_\lambda\Big)^{r^{\len(\lambda)-1}},
\]
which is the claimed identity. (No separate correction factor $\Res(y_1^r,\dots,y_{\len(\lambda)}^r)$ is needed here: Lemma~\ref{lem:powersubstitution}, applied with possibly unequal degrees $d_i'$, already produces the clean exponent $r^{\len(\lambda)-1}$ and nothing else, since its own proof normalizes $\Res(x_1^{d_1},\dots,x_n^{d_n})=1$ once and for all.)
\end{proof}

\begin{theorem}
\label{thm:main2}
Assume $n \geqslant 2$ and a system of $n$ homogeneous polynomials $f^{\{1\}},\dots,f^{\{n\}}$ in $\mathbb{C}[x_1,\dots,x_n]$ of the same degree $d$, equivariant with respect to the reflection group $G(r,n)$.
\begin{itemize}
\item If $d/r \geqslant n$ then
\[
\Res\big(f^{\{1\}},\dots,f^{\{n\}}\big) = \prod_{\lambda \vdash n} \Res\Big(\Delta f^{\{1\}}_\lambda,\dots,\Delta f^{\{1,\dots,\len(\lambda)\}}_\lambda\Big)^{m_\lambda \times r^{n-\len(\lambda)}}.
\]
\item If $d/r < n$ then
\[
\Res\big(f^{\{1\}},\dots,f^{\{n\}}\big) = \Big(\Delta f^{\{1,\dots,d/r+1\}}\Big)^{m_0} \times \prod_{\substack{\lambda\vdash n\\ \len(\lambda)\leqslant d/r}} \Res\Big(\Delta f^{\{1\}}_\lambda,\dots,\Delta f^{\{1,\dots,\len(\lambda)\}}_\lambda\Big)^{m_\lambda \times r^{n-\len(\lambda)}}
\]
where
\[
m_0 := nd^{n-1} - \sum_{\substack{\lambda\vdash n\\ \len(\lambda)\leqslant d/r}} m_\lambda\times r^{n-\len(\lambda)}\left(\sum_{j=1}^{\len(\lambda)}\frac{d(d-r)\cdots(d+(1-\len(\lambda))r)}{d+(1-j)r}\right).
\]
\end{itemize}
\end{theorem}

\begin{proof}
By Proposition~\ref{prop:deltaP} (which applies to every $\lambda \vdash n$ with $\len(\lambda) \leqslant d/r$, exactly the partitions occurring below),
\begin{equation}
\label{eq:step2}
\Res\Big(\Delta f^{\{1\}}_\lambda,\dots,\Delta f^{\{1,\dots,\len(\lambda)\}}_\lambda\Big) = \Res\Big(P^{\{1\}}_\lambda,\dots,P^{\{1,\dots,\len(\lambda)\}}_\lambda\Big)^{r^{\len(\lambda)-1}}.
\end{equation}
Since $\len(\lambda) \leqslant n$ for every $\lambda \vdash n$, the integer $r^{n-\len(\lambda)}$ is a nonnegative power of $r$, and raising both sides of \eqref{eq:step2} to that power gives
\begin{equation}
\label{eq:step3}
\Res\Big(\Delta f^{\{1\}}_\lambda,\dots,\Delta f^{\{1,\dots,\len(\lambda)\}}_\lambda\Big)^{r^{n-\len(\lambda)}} = \Res\Big(P^{\{1\}}_\lambda,\dots,P^{\{1,\dots,\len(\lambda)\}}_\lambda\Big)^{r^{n-1}}.
\end{equation}
Substituting \eqref{eq:step3} into the two formulas of Theorem~\ref{thm:main1} replaces each factor
\[
\Res\big(P^{\{1\}}_\lambda,\dots,P^{\{1,\dots,\len(\lambda)\}}_\lambda\big)^{m_\lambda \times r^{n-1}} \quad\text{by}\quad \Res\big(\Delta f^{\{1\}}_\lambda,\dots,\Delta f^{\{1,\dots,\len(\lambda)\}}_\lambda\big)^{m_\lambda\times r^{n-\len(\lambda)}},
\]
which is exactly the exponent claimed above.

It remains to identify the leading factor in the case $d/r < n$ and to determine $m_0$. Taking $I = \{1,\dots,d/r+1\}$ in the identity $P^I(x_1^r,\dots,x_n^r) = \Delta f^I(x_1,\dots,x_n)$ established at the start of the proof of Proposition~\ref{prop:deltaP}, and using that $P^I$ is a constant (it has degree $d/r-|I|+1 = 0$ by Lemma~\ref{lem:dividedwelldefined}, since here $|I|=d/r+1$), we get $\Delta f^{\{1,\dots,d/r+1\}} = P^{\{1,\dots,d/r+1\}}$ as the same constant. Hence, for any exponent $k$, $\big(\Delta f^{\{1,\dots,d/r+1\}}\big)^k = \big(P^{\{1,\dots,d/r+1\}}\big)^k$, and comparing the leading factors of Theorem~\ref{thm:main1} and of the statement above shows that we need $m_0 = r^{n-1}\times m_0^{\mathrm{Thm}\,\ref{thm:main1}}$, where
\[
m_0^{\mathrm{Thm}\,\ref{thm:main1}} := n\binom{d/r}{n-1} - \sum_{\substack{\lambda\vdash n\\ \len(\lambda)\leqslant d/r}} m_\lambda\left(\sum_{j=1}^{\len(\lambda)}\frac{(d/r)(d/r-1)\cdots(d/r-\len(\lambda)+1)}{d/r-j+1}\right)
\]
is the constant appearing in Theorem~\ref{thm:main1}. One checks directly, factoring out an $r$ from each of the $\len(\lambda)$ terms of the falling product $d(d-r)\cdots(d+(1-\len(\lambda))r) = r^{\len(\lambda)}\cdot (d/r)(d/r-1)\cdots(d/r-\len(\lambda)+1)$ and from the denominator $d+(1-j)r = r(d/r-j+1)$, that
\[
r^{n-\len(\lambda)}\left(\sum_{j=1}^{\len(\lambda)}\frac{d(d-r)\cdots(d+(1-\len(\lambda))r)}{d+(1-j)r}\right) = r^{n-1}\left(\sum_{j=1}^{\len(\lambda)}\frac{(d/r)(d/r-1)\cdots(d/r-\len(\lambda)+1)}{d/r-j+1}\right),
\]
which is precisely the identity $m_0 = r^{n-1}\times m_0^{\mathrm{Thm}\,\ref{thm:main1}}$ term by term.

Equivalently, and independently of this cross-check, we compare degrees with respect to the coefficients of the $f^{\{i\}}$'s. The resultant on the left side is homogeneous of degree $d^{n-1}$ with respect to the coefficients of each polynomial $f^{\{i\}}$, so it is homogeneous of degree $nd^{n-1}$ with respect to the coefficients of all the $f^{\{i\}}$, $i=1,\dots,n$. Given a partition $\lambda \vdash n$, $\len(\lambda)\leqslant d/r$, the polynomial $\Delta f^{\{1,\dots,j\}}_\lambda$, $1\leqslant j \leqslant \len(\lambda)$, has degree $d+(1-j)r$. Therefore, the resultant associated with the partition $\lambda$ is homogeneous with respect to the coefficients of the $f^{\{i\}}$'s of degree
\[
\sum_{j=1}^{\len(\lambda)} \frac{d(d-r)\cdots(d+(1-\len(\lambda))r)}{d+(1-j)r}.
\]
Finally, since $\Delta f^{\{1,\dots,d/r+1\}}$ is homogeneous of degree one in the coefficients of the $f^{\{i\}}$'s, we deduce the formula for $m_0$ displayed in the statement.
\end{proof}

\begin{example}
\label{ex:second}
Consider the system of three homogeneous polynomials
\[
\begin{cases}
f^{\{1\}} = ax_1^{18} + (b+c)x_1^6x_2^6x_3^6 + cx_1^{12}x_2^6 + cx_1^{12}x_3^6 \\
f^{\{2\}} = ax_2^{18} + (b+c)x_1^6x_2^6x_3^6 + cx_2^{12}x_1^6 + cx_2^{12}x_3^6 \\
f^{\{3\}} = ax_3^{18} + (b+c)x_1^6x_2^6x_3^6 + cx_3^{12}x_1^6 + cx_3^{12}x_2^6
\end{cases}
\]
This system is equivariant with respect to the reflection groups $G(6,3)$, $G(3,3)$, $G(2,3)$, and $G(1,3)$. To facilitate the computation of the resultant of this system, we use the reflection group $G(6,3)$, so that $d/r = 18/6 = 3$, and
\begin{align*}
\Res\big(f^{\{1\}},f^{\{2\}},f^{\{3\}}\big) &= \Res\big(\Delta f^{\{1\}}_{(3)}\big)^{m_{(3)}\times 36} \times \Res\big(\Delta f^{\{1\}}_{(2,1)},\Delta f^{\{1,2\}}_{(2,1)}\big)^{m_{(2,1)}\times 6} \\
&\quad\times \Res\big(\Delta f^{\{1\}}_{(1,1,1)},\Delta f^{\{1,2\}}_{(1,1,1)},\Delta f^{\{1,2,3\}}_{(1,1,1)}\big)^{m_{(1,1,1)}\times 1}.
\end{align*}
We compute
\[
\Delta f^{\{1,2\}} = \frac{\Delta f^{\{1\}}-\Delta f^{\{2\}}}{x_1^6-x_2^6} = ax_1^{12}+ax_1^6x_2^6+ax_2^{12}+cx_1^6x_2^6+cx_1^6x_3^6+cx_2^6x_3^6,
\]
\[
\Delta f^{\{1,3\}} = \frac{\Delta f^{\{1\}}-\Delta f^{\{3\}}}{x_1^6-x_3^6} = ax_1^{12}+ax_1^6x_3^6+ax_3^{12}+cx_1^6x_2^6+cx_1^6x_3^6+cx_2^6x_3^6,
\]
\[
\Delta f^{\{1,2,3\}} = \frac{\Delta f^{\{1,2\}}-\Delta f^{\{1,3\}}}{x_2^6-x_3^6} = ax_1^6+ax_2^6+ax_3^6,
\]
and, under $\rho_\lambda$,
\begin{align*}
\Delta f^{\{1\}}_{(3)} &= (a+b+3c)x_1^{18}, \qquad \Delta f^{\{1\}}_{(2,1)} = (a+c)x_1^{18}+(b+2c)x_1^{12}x_2^6, \\
\Delta f^{\{1,2\}}_{(2,1)} &= (a+c)x_1^{12}+ax_2^{12}+(a+2c)x_1^6x_2^6.
\end{align*}
Hence
\[
\Res\big(\Delta f^{\{1\}}_{(3)}\big) = a+b+3c, \qquad \Res\big(\Delta f^{\{1\}}_{(2,1)},\Delta f^{\{1,2\}}_{(2,1)}\big) = a^{12}(a+c)^6(a^2-ab-ac+b^2+2bc)^6,
\]
\[
\Res\big(\Delta f^{\{1\}}_{(1,1,1)},\Delta f^{\{1,2\}}_{(1,1,1)},\Delta f^{\{1,2,3\}}_{(1,1,1)}\big) = a^{216}(a-c)^{108}(a+b)^{72},
\]
so that
\[
\Res\big(f^{\{1\}},f^{\{2\}},f^{\{3\}}\big) = a^{432}(a-c)^{108}(a+c)^{108}(a+b)^{72}(a+b+3c)^{36}\big(a^2-ab-ac+b^2+2bc\big)^{108}.
\]
\end{example}

\section{Discriminant of a homogeneous polynomial invariant under $G(r,n)$}
\label{sec:discriminant}

In this section, we use Theorem~\ref{thm:main1} to develop a decomposition formula for the discriminant of an invariant homogeneous polynomial under the action of the reflection group $G(r,n)$.

Let $f \in \mathbb{C}[x_1,\dots,x_n]$ of degree $d$ be a homogeneous polynomial invariant under the reflection group $G(r,n)$. For all $(\xi^{i_1},\dots,\xi^{i_n};\sigma) \in G(r,n)$,
\begin{equation}
\label{eq:finv}
\big[(\xi^{i_1},\dots,\xi^{i_n};\sigma)f\big](x_1,\dots,x_n) := f(\xi^{i_{\sigma(1)}}x_{\sigma(1)},\dots,\xi^{i_{\sigma(n)}}x_{\sigma(n)}) = f(x_1,\dots,x_n).
\end{equation}
For $1 \leqslant p \leqslant n$, the $p$-th elementary symmetric polynomial in $x_1,\dots,x_n$ is
\[
e_p(x_1,\dots,x_n) := \sum_{i_1<i_2<\dots<i_p} x_{i_1}x_{i_2}\cdots x_{i_p}.
\]
From the action of $G(r,n)$ on polynomials, we see immediately that for $1 \leqslant p < n$ the polynomials $\sigma_p := e_p(x_1^r,\dots,x_n^r)$ are invariants of $G(r,n)$.

We have
\begin{equation}
\label{eq:fPdisc}
f(x_1,\dots,x_n) = P(x_1^r,\dots,x_n^r)
\end{equation}
where $P$ is a homogeneous polynomial in $\mathbb{C}[x_1,\dots,x_n]$ of degree $d/r$. We denote the partial derivatives of $P$ by $P^{\{i\}}(x_1,\dots,x_n) := \partial P/\partial x_i(x_1,\dots,x_n)$, $i=1,\dots,n$, and the partial derivatives of $f$ by $f^{\{i\}}(x_1,\dots,x_n) := \partial f/\partial x_i(x_1,\dots,x_n)$, $i=1,\dots,n$.

The discriminant of $f$ is defined, following \citep{Demazure2012}, by
\begin{equation}
\label{eq:discres}
d^{a(n,d)}\Disc(f) = \Res\big(f^{\{1\}},f^{\{2\}},\dots,f^{\{n\}}\big),
\end{equation}
where $a(n,d) := \dfrac{(d-1)^n-(-1)^n}{d} \in \mathbb{Z}$, and $\Disc(f)$ is homogeneous of degree $n(d-1)^{n-1}$.

\begin{theorem}
\label{thm:disc1}
Assume $n \geqslant 2$ and $d \geqslant 2$. With the above notation, there exists a polynomial $K(f)$, depending on the coefficients of $P$, such that:
\begin{itemize}
\item If $d/r > n$ then
\[
\left(\frac{d}{r}\right)^{a(n,d/r)\times r^{n-1}}\Disc(f) = \prod_{\lambda\vdash n} \Res\Big(P^{\{1\}}_\lambda,\dots,P^{\{1,\dots,\len(\lambda)\}}_\lambda\Big)^{m_\lambda\times r^{n-1}} \times K(f).
\]
\item If $d/r \leqslant n$ then
\[
\left(\frac{d}{r}\right)^{a(n,d/r)\times r^{n-1}}\Disc(f) = \Big(P^{\{1,\dots,d/r+1\}}\Big)^{m_0\times r^{n-1}} \times \prod_{\substack{\lambda\vdash n\\ \len(\lambda)<d/r}} \Res\Big(P^{\{1\}}_\lambda,\dots,P^{\{1,\dots,\len(\lambda)\}}_\lambda\Big)^{m_\lambda\times r^{n-1}} \times K(f)
\]
where
\[
m_0 := n\binom{d/r-1}{n-1} - \sum_{\substack{\lambda\vdash n\\ \len(\lambda)<d/r}} m_\lambda\left(\sum_{j=1}^{\len(\lambda)}\frac{(d/r-1)\cdots(d/r-\len(\lambda))}{d/r-j}\right).
\]
\end{itemize}
\end{theorem}

\begin{proof}
We have $\Disc(f(x_1,\dots,x_n)) = \Disc(P(x_1^r,\dots,x_n^r))$. By \citep[Proposition 4.15]{BuseJouanolou2014}, there exists a polynomial $K(f)$, depending on the coefficients of $P$, such that
\[
\Disc\big(f(x_1,\dots,x_n)\big) = \Disc\big(P(x_1,\dots,x_n)\big)^{r^{n-1}} \times \Res\big(x_1^r,\dots,x_n^r\big)^{\frac{d}{r}\left(\frac{d}{r}-1\right)^{n-1}} \times K(f).
\]
Since $\Res(x_1^r,\dots,x_n^r) = 1$ by the normalization of the resultant, this reduces to
\begin{equation}
\label{eq:discPstep}
\Disc\big(f(x_1,\dots,x_n)\big) = \Disc\big(P(x_1,\dots,x_n)\big)^{r^{n-1}} \times K(f).
\end{equation}
Since $P(x_1,\dots,x_n)$ is a homogeneous polynomial in $\mathbb{C}[x_1,\dots,x_n]$ of degree $d/r$, applying the discriminant--resultant relation \eqref{eq:discres} to $P$ gives
\begin{equation}
\label{eq:discPresP}
\left(\frac{d}{r}\right)^{a(n,d/r)}\Disc(P(x_1,\dots,x_n)) = \Res\big(P^{\{1\}},P^{\{2\}},\dots,P^{\{n\}}\big).
\end{equation}
Combining \eqref{eq:discPstep} and \eqref{eq:discPresP},
\begin{equation}
\label{eq:combined}
\left(\frac{d}{r}\right)^{a(n,d/r)\times r^{n-1}}\Disc(f(x_1,\dots,x_n)) = \Res\big(P^{\{1\}},\dots,P^{\{n\}}\big)^{r^{n-1}} \times K(f).
\end{equation}
For all $(\xi^{i_1},\dots,\xi^{i_n};\sigma) \in G(r,n)$, we have
\[
\big[(\xi^{i_1},\dots,\xi^{i_n};\sigma)f\big](x_1,\dots,x_n) := f(\xi^{i_{\sigma(1)}}x_{\sigma(1)},\dots,\xi^{i_{\sigma(n)}}x_{\sigma(n)}) = f(x_1,\dots,x_n),
\]
\[
\big[(\xi^{i_1},\dots,\xi^{i_n};\sigma)f\big](x_1,\dots,x_n) := P(x_{\sigma(1)}^r,\dots,x_{\sigma(n)}^r) = P(x_1^r,\dots,x_n^r).
\]
As in the proof of the Remark following the definition of $G(r,n)$-equivariance in Section~\ref{sec:resultant}, injectivity of $y_k \mapsto x_k^r$ lets us cancel the substitution in the last display, so $P(x_1,\dots,x_n)$ is a homogeneous symmetric polynomial.

Consequently, its partial derivatives $P^{\{1\}},\dots,P^{\{n\}}$ form an $S_n$-equivariant system: for $\sigma \in S_n$, differentiating $P(x_{\sigma(1)},\dots,x_{\sigma(n)}) = P(x_1,\dots,x_n)$ with respect to $x_i$ via the chain rule gives $P^{\{\sigma(i)\}}(x_{\sigma(1)},\dots,x_{\sigma(n)}) = P^{\{i\}}(x_1,\dots,x_n)$, i.e.\ $\sigma\big(P^{\{i\}}\big) = P^{\{\sigma(i)\}}$.

The decomposition formula for the resultant of an $S_n$-equivariant homogeneous polynomial system (\citep{BuseKarasoulou2016}) thus applies to $P^{\{1\}},\dots,P^{\{n\}}$, which have degree $d/r - 1$. Since $d/r > n \iff d/r-1 \geqslant n$ and $d/r \leqslant n \iff d/r - 1 < n$ (both equivalences using that $d/r$ and $n$ are integers), that decomposition reads, in the notation of Theorem~\ref{thm:main1}:
\begin{itemize}
\item If $d/r > n$ then $\displaystyle \Res\big(P^{\{1\}},\dots,P^{\{n\}}\big) = \prod_{\lambda \vdash n} \Res\Big(P^{\{1\}}_\lambda,\dots,P^{\{1,\dots,\len(\lambda)\}}_\lambda\Big)^{m_\lambda}$;
\item If $d/r \leqslant n$ then $\displaystyle \Res\big(P^{\{1\}},\dots,P^{\{n\}}\big) = \Big(P^{\{1,\dots,d/r+1\}}\Big)^{m_0} \times \prod_{\substack{\lambda\vdash n\\ \len(\lambda)<d/r}} \Res\Big(P^{\{1\}}_\lambda,\dots,P^{\{1,\dots,\len(\lambda)\}}_\lambda\Big)^{m_\lambda}$,
\end{itemize}
with $m_0$ as displayed in the statement (this is exactly $m_0$ of Theorem~\ref{thm:main1} with $d/r$ replaced by $d/r-1$). Raising each of these two identities to the power $r^{n-1}$ and substituting into \eqref{eq:combined} yields the two claimed formulas.
\end{proof}

\begin{proposition}
\label{prop:theta}
Let $f \in \mathbb{C}[x_1,\dots,x_n]$ of degree $d$ be a homogeneous polynomial invariant under the reflection group $G(r,n)$. For $r>1$ and for all $i=1,\dots,n$, there exists a homogeneous polynomial $\Theta f^{\{i\}}$ of degree $d-r$ such that
\[
\frac{\partial f}{\partial x_i}(x_1,\dots,x_n) = r x_i^{r-1}\, \Theta f^{\{i\}}(x_1,\dots,x_n).
\]
The polynomials $\Theta f^{\{1\}},\Theta f^{\{2\}},\dots,\Theta f^{\{n\}}$ form a $G(r,n)$-equivariant system of homogeneous polynomials of degree $d-r$.
\end{proposition}

\begin{proof}
Since $f$ is invariant under $G(r,n)$, $f(x_1,\dots,x_n) = P(x_1^r,\dots,x_n^r)$ where $P$ is homogeneous of degree $d/r$. For $r>1$ and $i=1,\dots,n$,
\begin{align*}
\frac{\partial f}{\partial x_i}(x_1,\dots,x_n) &= r x_i^{r-1}\, \frac{\partial P}{\partial x_i}(x_1^r,\dots,x_n^r) = r x_i^{r-1}\,\Theta f^{\{i\}}(x_1,\dots,x_n), \\
\Theta f^{\{i\}}(x_1,\dots,x_n) &:= \frac{\partial P}{\partial x_i}(x_1^r,\dots,x_n^r).
\end{align*}
As shown in the proof of Theorem~\ref{thm:disc1}, the partial derivatives $P^{\{1\}},\dots,P^{\{n\}}$ of the symmetric polynomial $P$ form an $S_n$-equivariant system of homogeneous polynomials of degree $d/r-1$. By the ``$\Leftarrow$'' direction of the Remark following the definition of $G(r,n)$-equivariance in Section~\ref{sec:resultant}, it follows that $P^{\{1\}}(x_1^r,\dots,x_n^r),\dots,P^{\{n\}}(x_1^r,\dots,x_n^r)$ form a $G(r,n)$-equivariant system of homogeneous polynomials of degree $d-r$.
\end{proof}

\begin{theorem}
\label{thm:disc2}
Let $f \in \mathbb{C}[x_1,\dots,x_n]$ of degree $d$ be a homogeneous polynomial invariant under $G(r,n)$. Assume $n\geqslant 2$, $d\geqslant 2$, $r>1$, and let $\Theta f^{\{i\}}$, $i=1,\dots,n$, be as in Proposition~\ref{prop:theta}. Then
\begin{multline*}
d^{a(n,d)}\Disc(f) = r^{n(d-1)^{n-1}} \times \Res\big(\Theta f^{\{1\}},\dots,\Theta f^{\{n\}}\big) \\
\times \prod_{i=1}^{n-1} \Res\big(\Theta f^{\{1\}},\dots,\Theta f^{\{i\}}, x_{i+1},\dots,x_n\big)^{\binom{n}{i}(r-1)^{n-i}}.
\end{multline*}
\end{theorem}

\begin{proof}
By \eqref{eq:discres} and Proposition~\ref{prop:theta},
\begin{align*}
d^{a(n,d)}\Disc(f) &= \Res\left(\frac{\partial f}{\partial x_1},\dots,\frac{\partial f}{\partial x_n}\right) = \Res\big(rx_1^{r-1}\Theta f^{\{1\}},\dots,rx_n^{r-1}\Theta f^{\{n\}}\big) \\
&= r^{n(d-1)^{n-1}}\Res\big(x_1^{r-1}\Theta f^{\{1\}},\dots,x_n^{r-1}\Theta f^{\{n\}}\big).
\end{align*}
By multiplicativity of the resultant, applied once in each of the $n$ slots, $\Res\big(x_1^{r-1}\Theta f^{\{1\}},\dots,x_n^{r-1}\Theta f^{\{n\}}\big)$ splits into a product of $2^n$ factors $\Res(h_1,\dots,h_n)$, indexed by the subsets $S \subseteq \{1,\dots,n\}$ of slots where $h_i := x_i^{r-1}$ is chosen (and $h_i := \Theta f^{\{i\}}$ for $i \notin S$).

These $2^n$ factors depend only on $|S|$. Fix $i := n-|S|$, the number of slots carrying a $\Theta f$, and set $D := (r-1)^{|S|}(d-r)^{n-|S|}$, the product of the degrees of $h_1,\dots,h_n$ (depending only on $|S|$, not on which slots are in $S$). Given two subsets $S,S'\subseteq\{1,\dots,n\}$ of the same size, choose $\sigma \in S_n$ with $\sigma(\{1,\dots,n\}\setminus S) = \{1,\dots,n\}\setminus S'$. Since $\Theta f^{\{1\}},\dots,\Theta f^{\{n\}}$ form a $G(r,n)$-equivariant system (Proposition~\ref{prop:theta}), they are in particular equivariant under the pure permutation part of $G(r,n)$, i.e.\ $\Theta f^{\{k\}}(x_{\sigma(1)},\dots,x_{\sigma(n)}) = \Theta f^{\{\sigma(k)\}}(x_1,\dots,x_n)$ for every $k$. Substituting $x_k \mapsto x_{\sigma(k)}$ in every entry of the tuple $(h_1,\dots,h_n)$ indexed by $S$ sends slot $i$ to $x_{\sigma(i)}^{r-1}$ (if $i\in S$) or to $\Theta f^{\{\sigma(i)\}}$ (if $i\notin S$); relabelling the slots by $j=\sigma(i)$ shows that, up to reordering the $n$ arguments by $\sigma$, this is exactly the tuple indexed by $\sigma(S)$, which equals $S'$. By the ``linear change of variables'' property, the substitution $x_k \mapsto x_{\sigma(k)}$ multiplies the resultant by $\varepsilon(\sigma)^D$; by the ``permutation of polynomials'' property, reordering the $n$ arguments by $\sigma$ multiplies it by the same factor $\varepsilon(\sigma)^D$ again. The two signs multiply to $\varepsilon(\sigma)^{2D}=1$, so
\[
\Res(h_1,\dots,h_n)\big|_{\text{indexed by }S} = \Res(h_1,\dots,h_n)\big|_{\text{indexed by }S'}.
\]
Consequently all $\binom{n}{i}$ subsets $S$ of size $n-i$ give the same value, which we compute using the canonical choice $S=\{i+1,\dots,n\}$:
\[
R_i := \Res\big(\Theta f^{\{1\}},\dots,\Theta f^{\{i\}}, x_{i+1}^{r-1},\dots,x_n^{r-1}\big), \qquad i=0,1,\dots,n,
\]
(with the convention that no $\Theta f$ appears when $i=0$, and no $x^{r-1}$ appears when $i=n$), so that
\[
\Res\big(x_1^{r-1}\Theta f^{\{1\}},\dots,x_n^{r-1}\Theta f^{\{n\}}\big) = \prod_{i=0}^n R_i^{\binom{n}{i}}.
\]
We have $R_0 = \Res(x_1^{r-1},\dots,x_n^{r-1}) = 1$ by the normalization of the resultant, and $R_n = \Res\big(\Theta f^{\{1\}},\dots,\Theta f^{\{n\}}\big)$.

\emph{Reduction of $R_i$ for $1\leqslant i \leqslant n-1$.} The map $(x_1,\dots,x_n) \mapsto (x_1,\dots,x_i,x_{i+1}^{r-1},\dots,x_n^{r-1})$ is the identity on the first $i$ coordinates and the power map $x\mapsto x^{r-1}$ on the last $n-i$; it defines a morphism $\mathbb{P}^{n-1}(\mathbb{C})\to\mathbb{P}^{n-1}(\mathbb{C})$ onto of degree $(r-1)^{n-i}$. Exactly the same argument as in the proof of Lemma~\ref{lem:powersubstitution} (same vanishing locus, proportionality by irreducibility of the universal resultant, degree count via the classical homogeneity property, and normalization by specializing $\Theta f^{\{k\}}=y_k$) gives
\begin{align*}
R_i &= \Res\big(\Theta f^{\{1\}},\dots,\Theta f^{\{i\}},x_{i+1}^{r-1},\dots,x_n^{r-1}\big) \\
&= \Res\big(\Theta f^{\{1\}},\dots,\Theta f^{\{i\}},x_{i+1},\dots,x_n\big)^{(r-1)^{n-i}}.
\end{align*}
Altogether,
\begin{align*}
\Res\big(x_1^{r-1}\Theta f^{\{1\}},\dots,x_n^{r-1}\Theta f^{\{n\}}\big) &= \prod_{i=1}^{n-1}\Res\big(\Theta f^{\{1\}},\dots,\Theta f^{\{i\}},x_{i+1},\dots,x_n\big)^{\binom{n}{i}(r-1)^{n-i}} \\
&\quad\times \Res\big(\Theta f^{\{1\}},\dots,\Theta f^{\{n\}}\big),
\end{align*}
and we deduce the claimed formula.
\end{proof}

\begin{remark}
For all $i=1,\dots,n-1$,
\[
\Res\big(\Theta f^{\{1\}},\dots,\Theta f^{\{i\}},x_{i+1},\dots,x_n\big) = \Res\big(\Theta f^{\{1\}}(x_1,\dots,x_i,0,\dots,0),\dots,\Theta f^{\{i\}}(x_1,\dots,x_i,0,\dots,0)\big).
\]
These resultants can be computed using Theorem~\ref{thm:main1} or Theorem~\ref{thm:main2}, since the polynomials $\Theta f^{\{1\}}(x_1,\dots,x_i,0,\dots,0),\dots,\Theta f^{\{i\}}(x_1,\dots,x_i,0,\dots,0)$, for $i=1,\dots,n-1$, form a $G(r,i)$-equivariant system of homogeneous polynomials.
\end{remark}

\begin{example}
Consider the homogeneous polynomial of degree $6$
\[
f := a x_1^6+ax_2^6+ax_3^6+ax_4^6 + b x_1^3x_2^3+bx_1^3x_3^3+bx_1^3x_4^3+bx_2^3x_3^3+bx_2^3x_4^3+bx_3^3x_4^3,
\]
which is invariant under the reflection group $G(3,4)$. Its partial derivatives are
\[
\begin{cases}
f^{\{1\}} = 6ax_1^5+3bx_1^2x_2^3+3bx_1^2x_3^3+3bx_1^2x_4^3 \\
f^{\{2\}} = 6ax_2^5+3bx_1^3x_2^2+3bx_2^2x_3^3+3bx_2^2x_4^3 \\
f^{\{3\}} = 6ax_3^5+3bx_1^3x_3^2+3bx_2^3x_3^2+3bx_3^2x_4^3 \\
f^{\{4\}} = 6ax_4^5+3bx_1^3x_4^2+3bx_2^3x_4^2+3bx_3^3x_4^2
\end{cases}
\qquad
\begin{cases}
\Theta f^{\{1\}} = 2ax_1^3+bx_2^3+bx_3^3+bx_4^3 \\
\Theta f^{\{2\}} = 2ax_2^3+bx_1^3+bx_3^3+bx_4^3 \\
\Theta f^{\{3\}} = 2ax_3^3+bx_1^3+bx_2^3+bx_4^3 \\
\Theta f^{\{4\}} = 2ax_4^3+bx_1^3+bx_2^3+bx_3^3
\end{cases}
\]
Since $r=3$ (so $r-1=2$) and $n=4$, the exponents in Theorem~\ref{thm:disc2} are $\binom{4}{1}(r-1)^3 = 32$, $\binom{4}{2}(r-1)^2=24$, $\binom{4}{3}(r-1)^1=8$, so
\begin{align*}
6^{\frac{5^4-1}{6}}\Disc(f) &= 3^{4\times 5^3} \times \Res\big(\Theta f^{\{1\}},\Theta f^{\{2\}},\Theta f^{\{3\}},\Theta f^{\{4\}}\big) \times \Res\big(\Theta f^{\{1\}},x_2,x_3,x_4\big)^{32} \\
&\quad \times \Res\big(\Theta f^{\{1\}},\Theta f^{\{2\}},x_3,x_4\big)^{24} \times \Res\big(\Theta f^{\{1\}},\Theta f^{\{2\}},\Theta f^{\{3\}},x_4\big)^{8}.
\end{align*}
We compute
\[
\Res\big(\Theta f^{\{1\}},\Theta f^{\{2\}},\Theta f^{\{3\}},\Theta f^{\{4\}}\big) = (2a+3b)^{27}(2a-b)^{81}, \qquad \Res\big(\Theta f^{\{1\}},x_2,x_3,x_4\big) = 2a,
\]
\[
\Res\big(\Theta f^{\{1\}},\Theta f^{\{2\}},x_3,x_4\big) = (2a-b)^3(2a+b)^3, \qquad \Res\big(\Theta f^{\{1\}},\Theta f^{\{2\}},\Theta f^{\{3\}},x_4\big) = 2^9(2a-b)^{18}(a+b)^9.
\]
Raising to the exponents $1,32,24,8$ respectively and simplifying (the powers of $2$ on both sides cancel exactly, since $6^{104} = 2^{104}3^{104}$ on the left matches $2^{32}\times 2^{72}$ on the right), we obtain
\[
\Disc(f) = 3^{396}\, a^{32}(a+b)^{72}(2a-b)^{297}(2a+b)^{72}(2a+3b)^{27}.
\]
\end{example}

\section{Conclusion and perspectives}
\label{sec:conclusion}

In this paper we studied polynomial systems equivariant with respect to the complex reflection group $G(r,n)$, and homogeneous polynomials invariant under it. In Section~\ref{sec:resultant}, extending the decomposition formula of Bus\'e and Karasoulou \citep{BuseKarasoulou2016} for $S_n$-equivariant systems, we obtained a decomposition formula for the resultant of a $G(r,n)$-equivariant homogeneous polynomial system, first in terms of the associated $S_n$-equivariant system $P^{\{1\}},\dots,P^{\{n\}}$ (Theorem~\ref{thm:main1}), and then, using divided differences, directly in terms of the original system $f^{\{1\}},\dots,f^{\{n\}}$ (Proposition~\ref{prop:deltaP} and Theorem~\ref{thm:main2}). In each case the resultant of the full system splits into a product of resultants of smaller, partition-indexed subsystems, which are considerably easier to compute; the worked examples of Section~\ref{sec:resultant} illustrate the resulting simplification in practice.

In Section~\ref{sec:discriminant}, we applied this decomposition to the discriminant of a $G(r,n)$-invariant homogeneous polynomial $f$, along two complementary routes. Combining our resultant decomposition with the composition formula of Bus\'e and Jouanolou \citep{BuseJouanolou2014} for the discriminant of $f = P(x_1^r,\dots,x_n^r)$ yields Theorem~\ref{thm:disc1}, which expresses $\Disc(f)$ in terms of the same partition-indexed resultants as in Section~\ref{sec:resultant}, up to a correction factor $K(f)$, inherited from \citep{BuseJouanolou2014}, that our Theorem~\ref{thm:disc1} does not make explicit. To obtain a fully explicit formula, we introduced in Proposition~\ref{prop:theta} the polynomials $\Theta f^{\{i\}}$ occurring in $\partial f/\partial x_i = rx_i^{r-1}\Theta f^{\{i\}}$, and showed in Theorem~\ref{thm:disc2} that $\Disc(f)$ decomposes explicitly, with no unspecified factor, as a product of resultants of the $\Theta f^{\{i\}}$'s and of coordinate functions, with multiplicities $\binom{n}{i}(r-1)^{n-i}$ governed by the combinatorics of the reflection group. The worked example of Section~\ref{sec:discriminant} confirms this formula numerically.

Several questions remain open. First, it would be worthwhile to make the factor $K(f)$ of Theorem~\ref{thm:disc1} fully explicit, for instance by relating it to the factors appearing in Theorem~\ref{thm:disc2}; this would give a direct, self-contained proof of Theorem~\ref{thm:disc1} that does not rely on \citep{BuseJouanolou2014}. Second, $G(r,n)=G(r,1,n)$ is only one family in the Shephard--Todd classification of complex reflection groups; extending the present decomposition to the groups $G(r,q,n)$, and more generally to other (in particular exceptional) complex reflection groups, is a natural next step. Third, on the computational side, the partition-indexed decomposition suggests a genuine algorithmic gain over a direct application of Macaulay's formula to the full system, in the spirit of the motivating examples discussed in the Introduction (such as symmetry-invariant vortex equilibria \citep{FaugereSvartz2012}); it would be interesting to quantify this gain precisely and to implement it for large $n$ and $d$. Finally, since the discriminant of $f$ controls the singularities of the hypersurface it defines, our explicit formula could be used to study how the $G(r,n)$-symmetry constrains the discriminant locus, and to identify invariant families of polynomials with prescribed (or excluded) types of singularities.

\section*{Compliance with Ethical Standards}
\subsection*{Funding}
No funding were received to support this study 
\subsection*{Conflicts of Interest}
The authors declare that there are no conflicts of interest.
\subsection*{Author contribution declaration}  All the authors contribute equally in the conception and writing of this paper
\subsection*{Data Availability Statements}
No data were used to support this study

\end{document}